\documentclass[10pt]{amsart}
\usepackage{amsmath,amssymb}
\usepackage[T1]{fontenc}
\usepackage{lmodern,textcomp}
\usepackage{color}
\usepackage{graphicx}
\usepackage{amsthm}
\usepackage[all]{xy}
\usepackage{verbatim}
\usepackage{caption}

\usepackage{pgf,tikz}
\usetikzlibrary{arrows}

\def\R{\mathbb{R}}
\def\C{\mathbb{C}}
\def\Z{\mathbb{Z}}
\def\N{\mathbb{N}}
\def\Q{\mathbb{Q}}
\def\P{{\mathbb P}}
\def\O{\mathcal{O}}
\def\L{\mathcal{L}}
\def\A{\mathcal{A}}
\def\oA{\overline{\mathcal{A}}}
\def\F{\mathcal{F}}

\def\oC{\overline{\mathcal{C}}}
\def\oM{\overline{\mathcal{M}}}
\def\M{{\mathcal{M}}}

\def\oH{\overline{\mathcal{H}}}
\def\H{{\mathcal{H}}}

\begin{document}

\theoremstyle{definition}
\newtheorem{mydef}{Definition}[section]
\newtheorem{mynot}[mydef]{Notation}
\newtheorem{example}[mydef]{Example}
\newtheorem{remark}[mydef]{Remark}
\newtheorem{problem}[mydef]{Problem}
\theoremstyle{plain}
\newtheorem{myth}[mydef]{Theorem}
\newtheorem{myconj}[mydef]{Conjecture}
\newtheorem{mypr}[mydef]{Proposition}
\newtheorem{prdef}[mydef]{Proposition-Definition}
\newtheorem{mylem}[mydef]{Lemma}
\newtheorem{mycor}[mydef]{Corollary}
\newtheorem{assumption}[mydef]{Assumption}

\title[Volumes and Siegel-Veech constants of $\mathcal{H}(2g-2)$]{Volumes and Siegel-Veech constants of $\mathcal{H}(2g-2)$ and Hodge integrals}
\author{Adrien Sauvaget}
\address{IMJ-PRG, Universit\'e Pierre et Marie Curie\\ 4 place Jussieu\\ 75005 Paris, France}
\email{adrien.sauvaget@imj-prg.fr}
\date{\today}

\keywords{Moduli space of curves, translation surfaces, Masur-Veech volumes, Hodge integrals}
\subjclass[2010]{14H15, 14N10, 30F30, 30F60, 14C17}
\maketitle

\begin{abstract} 
In the 80's H. Masur and W. Veech defined two numerical invariants of strata of abelian differentials: the volume and the Siegel-Veech constant. Based on numerical experiments, A. Eskin and A. Zorich proposed a series of conjectures for the large genus asymptotics of these invariants.  By a careful analysis of the asymptotic behavior of quasi-modular forms, D. Chen, M. Moeller, and D. Zagier proved that this conjecture holds for strata of differentials with simple zeros. 

Here, with a mild assumption of existence of a good metric, we show that the conjecture holds for the other extreme case,  i.e. for strata of differentials with a unique zero. Our main ingredient is the expression of the numerical invariants of these strata in terms of Hodge integrals on moduli spaces of curves.
\end{abstract}

\setcounter{tocdepth}{1}
\tableofcontents

\section{Introduction}

\subsection{The Hodge bundle and its stratification} 

Let $g$ and $n$ be nonnegative integers satisfying  $2g-2+n>0$. We denote by $\M_{g,n}$ (respectively $\oM_{g,n}$) the moduli space of smooth (respectively stable nodal) curves of genus $g$ with $n$ marked points. Let $\pi:\oC_{g,n}\to \oM_{g,n}$ be the universal curve and $\sigma_i : \oM_{g,n} \to \oC_{g,n}$ the sections associated to marked points for $1\leq i\leq n$.

The {\em Hodge bundle} $p:\oH_{g,n}\to \oM_{g,n}$
 is the rank $g$ vector bundle whose sheaf of sections is $R^0\pi_*(\omega_{\oC_{g,n}/\oM_{g,n}})$. Its total space is the space of tuples $(C,x_1,\ldots,x_n,\alpha)$: stable curves endowed with an abelian differential. We also denote by $p: \P\oH_{g,n}\to \oM_{g,n}$ its projectivization and by $p:\H_{g,n}\to \M_{g,n}$ the restriction of the Hodge bundle to the locus of smooth curves.

The total space of the Hodge bundle is stratified according to the orders of zeros of the differential. Let $\mu=(k_1,\ldots, k_n)$ be a partition of $(2g-2)$. We denote by $\H(\mu)\subset \H_g$ the locus of curves endowed with an abelian differential with zeros of order $k_1,\dots, k_n$ (here the zeros are not marked). The dimension of $\H(\mu)$ is $2g-1+n$ and the Hodge bundle is the disjoint union of the $\H(\mu)$ for all partitions $\mu$ of $2g-2$. 

The locus $\H(\mu)$ is invariant under the $\C^*$-action. We denote by $\P\H(\mu)$ its projectivization. Besides, we denote by $\oH(\mu)$ (respectively $\P\oH(\mu)$) the Zariski closure of $\H(\mu)$ in $\oH_g$ (respectively of $\P\H(\mu)$ in $\P\oH_g$). The space $\P\oH(\mu)$ is a compact (singular) DM stack that has been precisely described in~\cite{BCGGM}. 

In the present paper we are mainly interested in the strata $\H(2g-2)$ of differentials with a unique zero. 

\subsection{Mazur-Veech Volumes}

Fix $g,n$ and $\mu$ as above. Let $(C,\alpha)$ be a point in $\H(\mu)$. We denote by $x_1,\ldots, x_n$ the zeros of $\alpha$. Consider the relative cohomology group 
$$
H= H^1(C, \{x_1, \ldots, x_n\}, \Z).
$$
The space $H\otimes \C$ provides a system of local coordinates of $\H(\mu)$ at the neighborhood of $(C,\alpha)$ called the {\em period coordinates}. The transition maps between two such system of coordinates are given by matrices with integer coefficients.  Therefore the space $\H(\mu)$ is endowed with an affine structure and with a volume form $\nu$ on : in period coordinates, this volume form is  the Lebesgue volume form normalized in such a way that the lattice $H\otimes(\Z\oplus i\Z)$ has volume 1. 

We denote by $\H_1(\mu)\subset \H(\mu)$ (respectively $\H_{\leq 1}(\mu)$) the subspace defined  by
$$
\frac{i}{2}\int_C \alpha \wedge \overline{\alpha} = 1 \text{  (respectively $\leq 1$)}.
$$ 
The volume form on $\H(\mu)$ induces a form $\nu_1$ on $\H_1(\mu)$, by a disintegration of $\nu$ (see~\cite{EskKonZor} for definition). The total volume of $\H_{1}(\mu)$ for $\nu_1$ is finite (see~\cite{Mas} and~\cite{Vee}). This is the {\em Masur-Veech volume} (or simply the  {\em volume}) of $\H_1(\mu)$. We denote it by ${\rm Vol}(\mu)$. 

\subsection{Siegel-Veech constants}

The spaces $\H(\mu)$ and $\H_1(\mu)$ are endowed with an action of ${\rm SL}(2,\R)$. This action is defined in period coordinates: the group ${\rm SL}(2,\R)$ acts simultaneously on all coordinates. The diagonal sub-group 
$$
\bigg\{\begin{pmatrix}
e^{t} & 0\\
0& e^{-t}
\end{pmatrix}, t\in \R \bigg\}
$$
endows $\H_1(\mu)$ with an ergodic flow with finite measure. This action lifts to the real vector bundle whose fiber at $(C,\alpha)$ is given by $H^1(C,\R)$. This vector bundle is also  endowed with an equivariant measure. This set-up allows to define the $2g$ {\em Lyapunov exponents} of the vector bundle. These $2g$ invariants are of the following form: $\lambda_1\geq  \lambda_2\ldots \geq \lambda_g\geq 0\geq \lambda_{g+1}=-\lambda_g\geq  \ldots \geq \lambda_{2g}= -\lambda_1$ (see~\cite{EskKonZor}) . We define the {\em Siegel-Veech constant} of $\H_1(\mu)$ as 
$$
c_{\rm area}(\mu)=  \frac{3}{\pi^2} \left( \lambda_1+\ldots + \lambda_g - \frac{1}{12} \sum_{k_i\in \mu} \frac{k_i(k_i+2)}{k_i+1}\right).
$$
\begin{remark}
This definition of  $c_{\rm area}(\mu)$ is actually  a theorem (see~\cite{EskKonZor}). The Siegel-Veech constant has an inner geometrical interpretation in terms of the number of families of closed geodesics of a general curve in $\H(\mu)$.
\end{remark}

\subsection{Intersection numbers on strata of differentials}

In the present text, unless otherwise  mentioned, we consider cohomology classes with rational coefficients. We use the following cohomology classes in $\oM_{g,n}$:
\begin{itemize}
\item for all $1\leq i\leq n$, let $\L_i=\sigma_i^* \omega_{\oC_{g,n}/\oM_{g,n}}\to \oM_{g,n}$ be the line cotangent line bundle at the $i$-th marked point. We denote by $\psi_i=c_1(\L_i) \in H^2(\oM_{g,n})$; 
\item for all $1\leq i\leq g$, we denote by $\lambda_i = c_i(\oH_{g,n})\in H^{2i}(\oM_{g,n})$ (we will no longer mention Lyapunov exponents, therefore the notation $\lambda_i$ will only stand for these Chern classes);
\item $\delta_0\in H^2(\oM_{g,n})$ is the Poincar\'e-dual class of the divisor whose generic point is a curve with a self-intersecting node.
\end{itemize}
If no confusion arises, we use the same notation for classes in $H^*(\oM_{g,n})$ and their pull-back to $H^*(\P \oH_{g,n})$ under $p$. 

We denote by $L=\O(1)\to \P\oH_{g,n}$ the dual of the canonical  line bundle and by $\xi=c_1(L)$ the {\em canonical class} (beware that here canonical class does not refer to the determinant of the cotangent bundle). We recall that the splitting principle implies
$$
H^*(\P\oH_{g,n})\simeq H^*(\oM_{g,n})[\xi]/ (\xi^g+\lambda_1 \xi^{g-1}+ \ldots + \lambda_g).
$$
The top cohomology group $H^{2(4g-4+n)}(\P\oH_{g,n})$ is canonically identified with $\Q$ by Poincar\'e-duality. We consider the following intersection numbers 
$$\int_{\P\oH(\mu)} \! \! \! \xi^{2g-2+n} \in \Q.$$

We will see that the intersection number $\int_{\P\oH(\mu)} \!  \xi^{2g-2+n}$ vanishes if $\mu\neq (2g-2)$. We will denote by
$$a_g = (-1)^g \int_{\P\oH(2g-2)} \! \! \! \! \! \! \! \! \! \xi^{2g-1}.$$
The line bundle $\O(1)$ is endowed with a natural singular hermitian metric (see Section~\ref{sec:volint}). In all the paper we will make the same assumption as in~\cite{KonZor1}.
\begin{assumption}\label{assumption}
There exists a desingularization $\phi:X\to\P\H(\mu)$ such that the curvature form associated to the hermitian metric on $\phi^*\O(1)$ is good in the sense of~\cite{Mum2}
\end{assumption}
Under this assumption, we have the following relation between $a_g$ and ${\rm Vol}(2g-2)$.
\begin{mypr}\label{vgag}
For all $g\geq 1$, we have
\begin{equation}\label{foragvg}
{\rm Vol}(2g-2)=\frac{2(2\pi)^{2g}}{(2g-1)!} a_g.
\end{equation}
\end{mypr} 

\begin{remark}
The assumption~\ref{assumption} should be proved soon. The authors of~\cite{BCGGM} have announced the existence of the desingularization of $\P\oH(\mu)$. This desingularization is essentialy obtained by a refinement of the data defining the stratification of $\P\oH(\mu)$.
\end{remark}

Besides, for a general $\mu$,  there exists a cohomology class $\beta\in H^{2(2g-3+n)}(\P\oH(\mu), \R)$ such that
$$
c_{\rm area}(\mu)=-\frac{1}{4\pi^2}\cdot \frac{\int_{\P\oH(\mu)} \delta_0\wedge\beta}{\int_{\P\oH(\mu)} \xi\wedge\beta}.
$$
In the specific case of $\mu=(2g-2)$, under the Assumption~\ref{assumption} the cocycle $\beta$ is a multiple of $\xi^{2g-2}$ (see~\cite{KonZor1}). Therefore, if we denote by 
$$
d_g=(-1)^{g-1} \int_{\P\oH(2g-2)} \delta_0 \xi^{2g-2},
$$
then we have the relation 
\begin{equation}\label{exprSV}
c_{\rm area}(2g-2)=\frac{d_g}{4\pi^2\, a_g}.
\end{equation} 
\begin{remark} Volumes and Siegel-Veech constants are positive, therefore $a_g$ and $d_g$ are positive (this explains our sign convention).
\end{remark}

\subsection{Statement of the results.}

We define the following formal series 
\begin{eqnarray*}
\F(t)&=& 1+ \sum_{g>0} (2g-1) a_g t^{2g},\\
\Delta(t)&=& \sum_{g>0} (2g-1) d_g t^{2g},\\
\mathcal{S}(t)&=&\frac{t/2}{{\rm sin}(t/2)}.
\end{eqnarray*} 
The main theorem of the paper is the following.
\begin{myth}\label{th:main}
For all $g>0$, we have 
\begin{equation}\label{mainfor}
[t^{2g}] \mathcal{S}(t)= \frac{1}{(2g)!} [t^{2g}] \mathcal{F}(t)^{2g}.
\end{equation}
and 
\begin{equation}\label{mainforbis}
[t^{2g-2}] \mathcal{S}(t)= \frac{2}{(2g-1)!} [t^{2g}] \left(\Delta\cdot \mathcal{F}(t)^{2g-1}\right).
\end{equation}
(where the notation $[t^n]$ stands for the $n$-th coefficient of the formal series).
\end{myth}
In particular, the above equality implies that the $a_g$'s and $d_g$'s can be computed inductively using the coefficients of $\mathcal{S}$.


\subsection{Asymptotic behavior for large genera}

In the past few years, algebraic geometers started to study the large genus asymptotic behavior of numerical invariants associated to moduli spaces of curves. For example, M. Mirzakhani and P. Zograf  identified the large genus asymptotics of the Weil-Petersson volumes (see~\cite{MirZog}).  For strata of differentials, A. Eskin and A. Zorich proposed the following conjectures.
\begin{myconj}[see~\cite{EskZor}]\label{conj2}
Volumes of strata satisfy
$${{\rm Vol}(k_1,\ldots, k_n)}= \frac{4}{(k_1+1)(k_2+1)\ldots (k_n+1)} (1+ \epsilon_1(\mu) )$$ 
where $\underset{g\to \infty}{\lim}\left( \underset{\mu\vdash 2g-2}{\rm max} |\epsilon_1(\mu)| \right)=0$.
\end{myconj}
\begin{myconj}\label{conj3}
We have
$$c_{\rm area}(\mu)=\frac{1}{2} +\epsilon_2(\mu)$$
where $\underset{g\to \infty}{\lim}\left( \underset{\mu\vdash 2g-2}{\rm max} |\epsilon_2(\mu)| \right)=0$.
\end{myconj}

We will use the induction Formulas~\eqref{mainfor} and~\eqref{mainforbis}  to compute the asymptotic expansion of $a_g$'s and $d_g$'s. Then, using Formulas~\eqref{foragvg} and~\eqref{exprSV}, we will deduce the following result.
\begin{myth}\label{th:asympt}
For all $R\geq 1$, we have
$$
\frac{(2g-2)}{4} {\rm Vol}(2g-2)=  1- \frac{\pi^2}{12g} + \frac{24 \pi^2-\pi^4}{288g^2}+\ldots  +  \frac{c_R}{g^R} +O\left(\frac{1}{g^{R+1}}\right),
$$
and 
$$c_{\rm area}(2g-2)= \frac{1}{2}- \frac{1}{4g} +\ldots  +  \frac{c_R'}{g^R} +O\left(\frac{1}{g^{R+1}} \right),
$$
where the coefficients $c_k$ and $c_k'$ lie in $\Q[\pi^2]$ and can be effectively computed. In particular conjectures~\ref{conj2},~and~\ref{conj3} hold for $\H(2g-2)$.
\end{myth}
\begin{remark}
For $g\geq 4$, the space $\H(2g-2)$ has three connected components: {\em hyperelliptic}, {\em odd} and {\em even} (see~\cite{KonZor}). Our formulas do not separate the volumes and Siegel-Veech constants of these three components. The volume is the sum of volumes of the connected components and the Siegel-Veech constant is a mean of the Siegel-veech constants of the connected components that is weighted by the volumes. 

 The hyperelliptic component has an explicitly computable volume and Siegel-Veech constants (see~\cite{AtrEskZor}). We will check that the volume of the hyperelliptic component is asymptotically negligible in comparison to the total volume. It remains to separate the volumes (and Siegel-Veech constants) of the odd and even  components. These are conjectured to be equivalent as $g$ goes to infinity (see~\cite{EskZor}).
\end{remark}

\subsection{Some comments on Theorem~\ref{th:asympt}}  In~\cite{CheMoeZag}, D. Chen, M. M\"oller and D. Zagier computed the asymptotic expansion of volumes and Siegel-Veech constants of strata of differentials with simple zeros:
$$
{\rm Vol}\big(\underset{2g-2}{\underbrace{1,1,\ldots,1}}\big) \underset{g\to +\infty}{\sim} \frac{1}{4^{g-1}}\left(1-\frac{\pi^{2}}{24g}+\frac{60\pi^2-\pi^4}{1152 g^2}+\ldots  \right), 
$$
and
$$  c_{\rm area}(1,\ldots,1)=\frac{1}{2}-\frac{1}{8g}
+\ldots.
$$
Therefore  Conjectures~\ref{conj2} and~\ref{conj3} also hold for these strata.

Their proof relies on a careful analysis of the formula of Eskin and Okounkov for volumes of strata. The main ingredient of this formula is the generating function of the number of ramified coverings of the punctured torus. These formal series are quasi-modular forms and the volumes (and Siegel-Veech constants) of strata are expressed using the asymptotic of the coefficients (see~\cite{EskOko}).

By numerical experiments, one observes that, for a fixed value of $g$, $|\epsilon_1(\mu)|$ and $|\epsilon_2(\mu)|$ are maximal for $\mu=(2g-2)$ and minimal for $\mu=(1,\ldots,1)$ (see Figure 1). 
\begin{center}
\begin{figure}[h]
\includegraphics[scale=0.36]{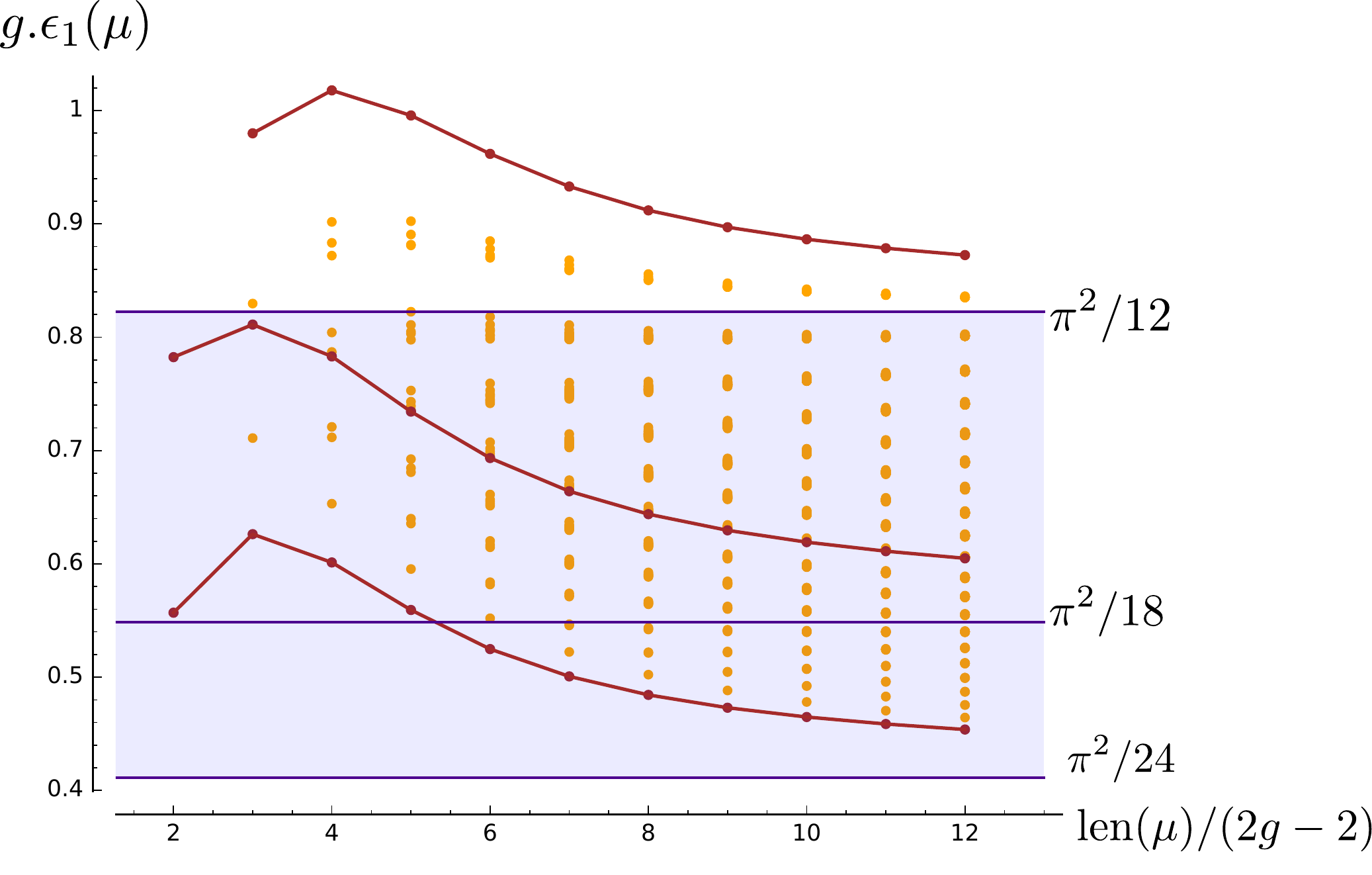} 
\caption{Absolute value of $g.\epsilon_1(\mu)$ in function of $g$. The broken lines correspond respectively to $\mu=(2g-2), (2,2,\ldots,2)$, and $(1,\ldots,1)$.}
\end{figure}
\end{center}

We can observe that the dominating term of $\epsilon_1(2g-2)$ is  $-\pi^2/12g$ which is twice the leading term in the expansion of $\epsilon_1(1,\ldots,1)$ (the same holds for $\epsilon_2$). This leads us to the following straightening of the conjecture of Eskin and Zorich (see Figure~2).
\begin{myconj}
The functions $\epsilon_1$ and $\epsilon_2$ satisfy
\begin{eqnarray*}
-\epsilon_1(\mu) &=& \frac{\pi^2 }{6 \; {\rm dim}(\H(\mu))} (1 +\epsilon_1'(\mu)), \\
-\epsilon_2(\mu) &=& \frac{1}{2 \; {\rm dim}(\H(\mu))} (1 +\epsilon_2'(\mu)),
\end{eqnarray*}
where $\epsilon_1'$ and and $\epsilon_2'$ tends uniformly to $0$ as $g$ goes to infinity.
\end{myconj}

\begin{center}
\begin{figure}[h]
 \includegraphics[scale=0.36]{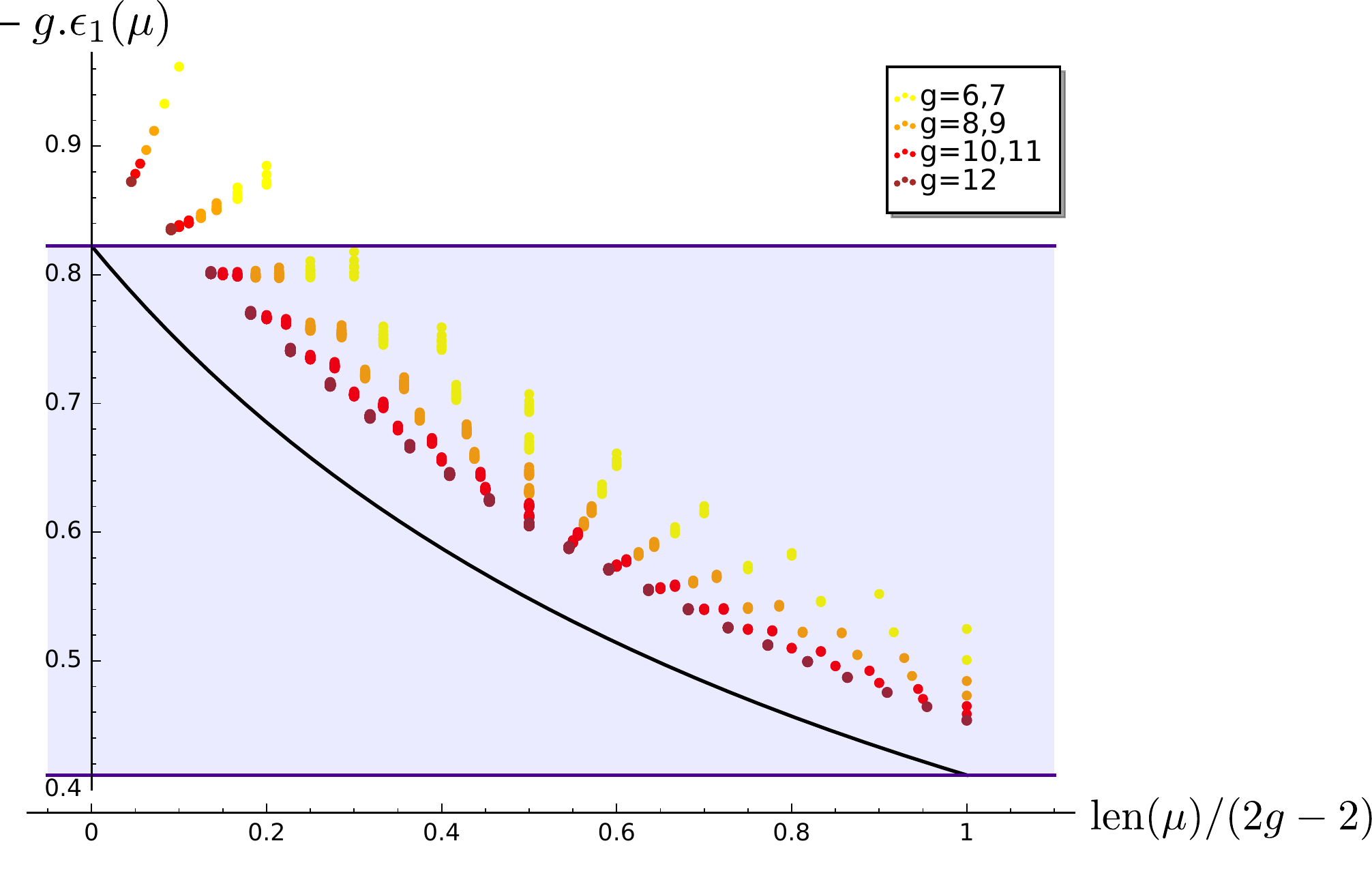} 
\caption{Absolute value of $g.\epsilon_1(\mu)$ in function of ${\rm len}(\mu)/(2g-2)$ for $g=6,7,\ldots,12$. In black, the graph of $y=\pi^2/12(1+x)$.}
\end{figure}
\end{center}


\subsection{Plan of the paper} We will follow linearly the general strategy of the introduction. In Section~\ref{sec:volint} we show how to express the volumes of the strata $\H_1(2g-2)$ in terms of integrals of $\xi$-classes. In Section~\ref{sec:int} we prove the induction formula for the integrals of $\xi$-classes (Theorem~\ref{th:main}). The main ingredient in this proof is the computation of the cohomlogy classes Poincar\'e-dual to $[\P\oH(\mu)]\in H^*(\P\oH_g)$ as in~\cite{Sau} (we will recall a simplified version of this computation here). Finally, in Section~\ref{sec:asympt} we analyze the asymptotic behavior of the integrals of $\xi$-classes to deduce Theorem~\ref{th:asympt}.

\subsection*{Acknowledgement} I would like to thank Dimitri Zvonkine, Martin Mo\"eller, Dawei Chen, Xavier Blot, Siarhei Finski and Felix Janda for very useful conversations on intersection of tautological classes over spaces of differentials. I am also very thankful to Anton Zorich and Charles Fougeron for having introduced me to the topic of large-genus invariants (and for having provided tables of numerical computations that allowed me to understand part of the results of the paper). Finally, I am very grateful to Elie de Panafieu for his precious help to handle the asymptotic analysis.

\section{Volumes and integrals of canonical classes}\label{sec:volint}

In this section we prove that ${\rm Vol}(2g-2)=\frac{2(2\pi)^{2g}}{(2g-1)!} a_g$ for all $g\geq 1$ (i.e. Proposition~\ref{vgag}). This identity may be known to experts. However, we could not find a reference to cite so we give the details of the proof. 

\subsection{The symplectic affine structure on $\H(2g-2)$}

Let $g\geq 1$. Let $(C,\alpha)\in \H(2g-2)$. Let $x\in C$ be the unique zero of $\alpha$. We have seen that a local parametrization of  $\H(2g-2)$ is given by the relative cohomology group $H^1(C,\{x\},\C)$. This space is isomorphic to $H^1(C,\C)$. Moreover, this space contains the lattice $H^1(C,\Z\oplus i\Z)$. 

We choose a symplectic basis of $H_1(C,\Z)$ made of closed curves  $(A_i,B_i)_{1\leq i\leq g}$ on $C$.  With this basis, the coordinates on $H^1(C,\C)$ are given by 
$$z_{A_i}= \int_{A_i}\!\!\! \alpha \hspace{5pt} \text{ and }\hspace{5pt} z_{B_i}=\int_{B_i}\!\!\!  \alpha.$$ 
Therefore the space $\H(2g-2)$ is endowed with an affine  structure whose transition map are matrices in ${\rm Sp}(2g,\Z)$. Besides, the reciprocity law defines a hermitian metric on $H^1(C,\C)$ given by
\begin{equation}\label{hermitian}
\langle \alpha, \alpha'\rangle = \frac{i}{2} \int_{C} \alpha \wedge \alpha'= \frac{i}{2} \sum_{i=1}^{g} (z_{A_i}  \overline{z}_{B_i}' -\overline{z}_{A_i}'  z_{B_i}) . 
\end{equation}
The lattice $H^1(C,\Z\oplus i\Z)\subset H^1(C,\C)$ defines a volume form $\nu$ in $H^{4g}(\oH(2g-2),\R)$ (normalized in such a way that the volume of a unit cube is $1$). Both the volume form and the hermitian metric are independent of the choice of the basis $(A_i,B_i)_{1\leq i\leq g}$. Now, for all $R\in \R_{+}$, we denote by $\H_{\leq R}(2g-2)$ the space of $(C,\alpha)$ with $||\alpha|| \leq R$.  

\subsection{Two volume forms on $\P\H(2g-2)$}  Using the period coordinates we define the following volume forms.
\begin{itemize}
\item  Let us consider the complex projectivization of the space of period coordinates $\P H^1(C,\C)$. The line bundle $\O(-1)\to \P H^1(C,\C)$ is endowed with the hermitian metric $h$ induced by the above hermitian metric $\langle\cdot,\cdot \rangle$ on $H^1(C,\C)$. This hermitian metric extends to a singular hermitian metric on $\O(-1)\to \P\oH(2g-2)$. We define $\omega$ to be the curvature form associated to $h$ in $\P\H(2g-2)$, i.e.
$$\omega=\frac{1}{2i\pi} \partial\overline{\partial}\log\left(h( \sigma)\right)
$$ 
for any local holomorphic non-vanishing section $\sigma$ of $\O(-1)$. With this 2-form we define $\nu_h=\omega^{2g-1}$. 

The form $\omega$ is a closed 2-form on $\P\H(2g-2)$. However, this form cannot be extended to a regular form on the whole boundary of $\P\oH(2g-2)$. Indeed, the function $h$ is  ill-defined along the divisor of pairs $(C,\alpha)$ with infinite area; i.e. when $C$ has at least one non-seperating node and the form $\alpha$ has a pole of order exactly $1$ at the two branches of this node.  This obstruction is the motivation for Assumption~\ref{assumption}.
\item 
The second volume form on $\P\H(2g-2)$, $\widetilde{\nu}$, is obtained by disintegration of $\nu$. We have a projection map  $p:\H(2g-2)\to \P\H(2g-2)$ whose fibers are isomorphic to $\C^*$. Let $D\subset \P\H(2g-2)$ be an open domain. Then the volume of $D$ for $\widetilde{\nu}$  is the total volume of $p^{-1}(D)\cap \H_{\leq 1}(2g-2)$.
\end{itemize}
\begin{mylem}\label{lem:relforms}
We have
\begin{equation}\label{relforms}
\nu_h = - \frac{ (2g-1)!}{2(2i\pi)^{2g}} \dim_\R(\H(2g-2)) \widetilde{\nu}.
\end{equation}
\end{mylem} 
\begin{proof}
Let us fix a point $(C,\alpha)\in \H(2g-2)$ (respectively in $\P\H(2g-2)$). An open neighborhood of this point is of the form is of the form $U/{\rm Aut}(C,\alpha)$ for some open set $U$ in $H^1(C,\C)$ (respectively in $\P H^1(C,\C)$). This open set is included in  the positive cone 
$$\mathcal{C}=\{v\in H^1(C,\C), \text{ s.t. } h(v)>0 \}$$ 
(respectively in $\P\mathcal{C}$). The volume forms $\nu_h$ and $\widetilde{\nu}$ can be defined on $\P\mathcal{C}$ and we will prove that the relation~\eqref{relforms} holds in $\P\mathcal{C}$.

In $H^1(C,\C)$ we have the coordinates $(z_{A_i}, z_{B_i})_{1\leq i\leq g}$ obtained from the symplectic structure on $C$. We introduce the following coordinates
\begin{eqnarray*}
z_{a_i}= \frac{1}{{2}} (z_{A_i}- i z_{B_i}),\\
z_{b_i}= \frac{1}{{2}}  (z_{A_i}+i  z_{B_i}).
\end{eqnarray*}
With these coordinates, the hermitian metric $h$ is given by 
$$
h(z_{a_i},z_{b_i})=\sum_{i=1}^g z_{a_i}\overline{z}_{a_i} - z_{b_i}\overline{z}_{b_i}
$$
and it has signature $(g,g)$. We will also consider the standard metric on $H^1(C,\C)$ given by
$$
h_{\rm st}(z_{a_i},z_{b_i})=\sum_{i=1}^g |z_{a_i}|^2 + |z_{b_i}|^2.
$$
Using this hermitian metric, we can define the curvature of $\O(-1)$ on $\P H^1(C,\C)$ using the Laplace-Beltrami  operator and the top intersection of this form $\nu_{\rm st}$. Besides one can also construct $\widetilde{\nu}_{\rm st}$ by the same procedure as $\widetilde{\nu}$. The volume forms $\nu_{\rm st}$ and $\widetilde{\nu}_{\rm st}$ are proportional (see \cite{Vois1}, Chapter 3 for example). Let us define the two following functions on $\P H^1(C,\C)$ with value in $\R_{>0}$.
$$
f= \frac{\nu_{h}}{\nu_{\rm st}}, \text{ and } \widetilde{f}= \frac{\widetilde{\nu}}{\widetilde{\nu}_{\rm st}}.
$$

We have two groups $U(g,g)$ and $U(2g)$ (for the standard metric) acting on $\P H^1(C,\C)$. The functions $f$ and $\widetilde{f}$ are invariant under the action of $\gamma\in U(g,g)\cap U(2g)$. Indeed $\nu_h$ is $U(g,g)$-equivariant and $\nu_{\rm st}$ is $U(2g)$-equivariant thus $f$ is invariant under $U(g,g)\cap U(2g)$. Besides the function $\widetilde{f}$ is equal to
$$
\widetilde{f} (v)= \left(\frac{h(v)}{h_{\rm st}(v)}\right)^{2g}
$$
for any vector $v\neq 0$  which is invariant under the action of $U(g,g)\cap U(2g)$. Therefore $\nu_h/\widetilde{\nu}$ is invariant under the action of $U(g,g)\cap U(2g)$.

The group $U(g,g)\cap U(g)$ contains the matrices of the form
$$
\left(
   \begin{array}{c | c}
      U &  0\\
      \hline
      
      0 & U' 
   \end{array}
   \right), \text{ with $U$ and $U'$ in $U(g)$.}
$$
For all $a$ and $b$ in $\R_{\geq 0}$, these matrices act transitively on the subspace
$$
H^1(C,\C) \supset E_{a,b}=\left\{ (z_{a_1},z_{b_1},\ldots), \sum_{i=1}^g |z_{a_i}|^2 = a \text{ and } \sum_{i=1}^g |z_{b_i}|^2 = b\right\}.
$$
Therefore, we only need to compare the volume forms $\nu_h$ and $\widetilde{\nu}$ at the points of the form $(z_{a_1},z_{b_1},0,0,\ldots,0).$

We consider the chart $U_{a_1}\subset\P H^1(C,\C)$ defined by $z_{a_1}\neq 0$ with its natural identification $U_{a_1}\simeq \C^{2g-1}$. The line bundle $\O(-1)$ has a natural section $\sigma_{a_1}$ over $U_{a_1}$ given by $\sigma_{a_1}(z_{b_1},z_{a_2},z_{b_2},\ldots)=(1,z_{b_1},z_{a_2},z_{b_2},\ldots)$. In this chart, the volume form $\widetilde{\nu}$ is given by

$$
\frac{2(2\pi)}{\dim_\R(\H(2g-2))\cdot h(\sigma_{a_1})^{2g}  } i^{2g-1} \left( dz_{b_1}\wedge d\overline{z}_{b_1} \wedge \prod_{i>1} (dz_{a_i}\wedge d\overline{z}_{a_i} \wedge dz_{b_i}\wedge d\overline{z}_{b_i})\right).
$$
Note the factor $2$ in this formula, it comes from the choice of coordinates $(z_{a_i},z_{b_i})$: in coordinates $(z_{A_i},z_{B_i})$ the expression of $\widetilde{\nu}$ is the same without this factor $2$.  At the point $(1,z_{b_1},0,\ldots,0)$, the 2-form $\omega$ is given by 
\begin{eqnarray*}
\omega \!\!\!\!\! &=&  \!\!\!\!\!  \frac{1}{2i\pi}  \frac{(1-|z_{b_1}|^2) \left(-  d{z}_{b_1}\wedge d\overline{z}_{b_1}+\sum_{i>1} d{z}_{a_i}\wedge d\overline{z}_{a_i}- d{z}_{b_i}\wedge d\overline{z}_{b_i}\right) - |z_{b_1}|^2  d{z}_{b_1}\wedge d\overline{z}_{b_1}}{ (1-|z_{b_1}|^2) ^2}\\
&=&  \!\!\!\!\!  \frac{1}{2i\pi} \left( \frac{\sum_{i>1} d{z}_{a_i}\wedge d\overline{z}_{a_i}- d{z}_{b_i}\wedge d\overline{z}_{b_i}}{ h(\sigma_{a_1}) } -  \frac{d{z}_{b_1}\wedge d\overline{z}_{b_1}}{ h(\sigma_{a_1})^2 } \right).
\end{eqnarray*}
We get 
\begin{eqnarray*}
\nu_h \!\!\!\!\! &=& \frac{(2g-1)!}{(2i\pi)^{2g-1} h(\sigma_{a_1})^{2g}} \cdot\left(-1\right)^{g} \left( dz_{b_1}\wedge d\overline{z}_{b_1} \wedge \prod_{i>1} (dz_{a_i}\wedge d\overline{z}_{a_i} \wedge dz_{b_i}\wedge d\overline{z}_{b_i})\right)\\
&=& - \frac{ (2g-1)!}{2(2i\pi)^{2g}} \dim_\R(\H(2g-2)) \widetilde{\nu}.
\end{eqnarray*}
\end{proof}
\begin{proof}[End of the proof of Proposition~\ref{pr:expr} under the assumption~\ref{assumption}] Let $\phi:X\to \P\oH(2g-2)$ be a desingularization of $\P\oH(2g-2)$ such that the hermitian metric on $\phi^*(\O(1))$ is good. Then the intersection number $\int_X \phi^*(\xi^{2g-1})$  is equal to the total volume of $\P\H(2g-2)$ for  $\nu_h=\omega^{2g-1}$. Now using the projection formula and the fact that $\phi$ is birational (thus of degree 1) we get 
$$\int_{\P\oH(2g-2)} \xi^{2g-1}=
\nu_h\left(\P\H(2g-2)\right).$$ Therefore Propostion~\ref{pr:expr} follows from Lemma~\ref{lem:relforms}.
\end{proof}

\

\section{Induction formula for integrals of canonical classes}\label{sec:int}

Let $Z=(k_1,\ldots,k_n)$ be a vector of non-negative integers. From now on in the text, for all vectors of integers we set
$$
\ell(Z)={\rm length}(Z) \hspace{5pt} \text{and} \hspace{5pt} |Z|=\sum_{i=1}^n k_i.
$$
\begin{mydef} The {\em projectivized marked stratum of type $Z$} is the locus $\A_g(Z)\subset \P\H_{g,n}$ defined as
$$
\left\{ (C,\alpha, x_1,\ldots, x_n),  \text{ s.t. $x_i$ is a zero of order $k_i$ for all $1\leq i\leq n$} \right\}.
$$ 
It is a smooth substack of $\P\H_{g,n}$ codimension $|Z|$. We denote by $\oA_g(Z)$ the Zariski closure of $\A_{g}(Z)$ in $\P\oH_{g,n}$.
\end{mydef}

Note that in this definition $Z$ does not need to satisfy $|Z|=2g-2$. However, if $|Z|=2g-2$, then we have
$$
\int_{\P\oH_g} \xi^{2g-2+n} [\P\oH(Z)] = \frac{1}{|{\rm Aut}(Z)|} \int_{\P\oH_{g,n}} \! \! \! \xi^{2g-2+n} [\oA_g(Z)] 
$$
where $Aut(Z)$ is the group of permutation of $[\![ 1,n]\!]$ preserving $Z$ (this follows from the projection formula). The purpose of this section is to compute the intersection number on the right-hand side.

\subsection{Vanishing for $n>1$}

First let us recall the following classical result.
\begin{mylem}[Mumford, \cite{Mum}] \label{lemmum} 
We have the following equality in $H^*(\P\oM_{g,n})$:
$$s_*(\oH_{g,n}) = c_*(\oH_{g,n})^{-1} = c_*(\oH_{g,n}^\vee)
$$
where $s_*$ and $c_*$ stand for the total Segre and Chern classes.
In particular $\lambda_g^2=0$.
\end{mylem}
We use this identity here to simplify the computation of $\xi^{2g-2+n} [\oA_g(Z)]$. The class $[\oA_g(Z)]$ is equal to
$$
\sum_{k=0}^{2g-2} \xi^{2g-2-k} \alpha^k_g(Z)
$$
where the classes $\alpha^k_g(Z)$  are pull-back from $H^{2k}(\oM_{g,n})$. Therefore the push-forward of $\xi^{2g-2+n}[\oA_g(Z)]$ under $p$ is given by
$$
\sum_{k=0}^{2g-2} s_{g-1+n+k} \alpha^k_g(Z)= \sum_{k=0}^{2g-2} (-1)^{g-1+n+k} \lambda_{g-1+n+k} \alpha^k_g(Z)
$$
(this follows from the projection formula and Lemma~\ref{lemmum}). However $\lambda_{2g-2+n+k}=0$ for $2g-2+n+k>g$. Therefore we get 
\begin{mypr}
The class $\xi^{2g-2+n} [\oA_g(Z)] $ vanishes if $n>1$. For $n=1$ we have:
$$
\int_{\P\oH_{g,1}} \! \! \! \xi^{2g-1} [\oA_g(2g-2)]  = \int_{\oM_{g,1}} \! \! \! (-1)^{g}  \lambda_g \alpha^0_g(Z).
$$
\end{mypr}

\subsection{$\lambda_g$-Conjecture}

We recall the following important result of Faber and Pandharipande.
\begin{mypr}[$\lambda_g$-conjecture, \cite{FabPan}]\label{lambdag}
Let $b_g=\int_{\oM_{g,1}} \! \! \! \lambda_g\psi_1^{2g-2}$. Then we have
$$
b_g= \frac{2^{2g-1}-1}{2^{2g-1}}\frac{|B_{2g}|}{(2g)!},
$$
where $B_{2g}$ is the $(2g)$th Bernouilli number. 
\end{mypr}
In particular, we have
$$ \mathcal{S}(t)= 1+ \sum_{g>0}  b_g t^{2g}
$$
(where we recall that $\mathcal{S}(t)=\frac{t/2}{\sin(t/2)}$).

\subsection{Stable differentials}

The main tool to prove Theorem~\ref{th:main} will be the induction formula established in~\cite{Sau} to compute the cohomology classes $[\overline{A}_{g}(Z)]$ in $H^*(\P\oH_{g,n})$. We will state a simplified version of this induction formula because we only need to compute the class $\lambda_g\alpha^0_g(Z)\in H^*(\oM_{g,n})$.  The notation of the present text will also be slightly lighter than the one of~\cite{Sau}.

\begin{mydef}
Let $P=(p_1,\ldots,p_m)$ be a vector of positive integers. The {\em space of stable differentials} $\oH_{g,n,P}$ is the space whose geometric points are tuples $(C,x_1,\ldots, x_{n+m},\alpha)$ where 
\begin{itemize}
\item $(C,x_1,\ldots, x_{n+m})$ is a pre-stable curve (i.e. a nodal curve with distinct marked points in the smooth locus);
\item $\alpha$ is a meromorphic differentials with poles of order $(p_i+1)$ at $x_{n+i}$ for all $1\leq i\leq m$;
\item there are finitely many automorphisms of $C$ preserving $\alpha$. 
\end{itemize} 
\end{mydef}
Let $p  :\oH_{g,n,P}  \to \oM_{g,n+m}$ be the forgetful map of the differential. The space $\oH_{g,n,P}$ is naturally equipped  with a structure of cone over $\oM_{g,n+m}$. In particular it has a projectivization $\P\oH_{g,n,P}$. The rank of this cone is $m+|P|+g-1$ (if $P$ is not empty).

\begin{mydef} Let $Z=(k_1,\ldots,k_n)$ be a vector of non-negative integers. We denote by $\mathcal{A}_{g}(Z,P)\subset \P\H_{g,n,P}$ the locus of differentials with zeros of order $k_i$ at $x_i$ for all $1\leq i\leq n$ and residues equal to zero at the poles. We denote by $\oA_{g}(Z,P)$ its Zariski closure.
\end{mydef}

Later, we will need the following lemma in genus 0.
\begin{mylem}
We suppose that $g=0$ and $Z=(k)$. If $k+1\neq  |P|$, then we have
$$p_*\left[\oA_0\left((k),P\right)\right]=0\in H^*(\oM_{0,m+1}).$$ 
If $k+1= |P|$, then this class is equal to $1\in H^0(\oM_{0,m+1})$.
\end{mylem}

\begin{proof}
The locus $\oA_0((k),P)$ is of codimension $k+m-1$ in $\P\oH_{0,1,P}$ which is of relative dimension $m+|P|-2$ over $\oM_{0,1+m}$. Thus the class $p_*[\oA_0((k),P)]$ vanishes if $k+1< |P|$.

Besides, any meromorphic differential without residues on a genus 0 curve is exact. In other words any such differential can be integrated to get a meromorphic function of degree $d=|P|$. Thus $\A_0((k),P)$ is empty if $k-1>d=| P|$.

Finally, if $k-1= | P|$ then $p_*[\A_{0}((k),P)]=1$. Indeed, for every curve in $\M_{0,1+m}$ there exists exactly one differential up to scale with a zero of order $k$ at $x_1$ and poles (without residue) of orders prescribed by $P$ at the other marked points: it is the derivative of the function
$$
\frac{1}{\prod_{i=1}^m (z-z_{i+1})^{p_i}}
$$
(the first marked point has coordinate $z_1=\infty$).
\end{proof}

\subsection{Twisted graphs}

Let $m$ be a positive integer. We denote by  ${\rm Part}(g)_m$ the set of vectors ${\underline{g}}= (g_0,g_1\ldots,g_m)$ of nonnegative integers such that $|{\underline{g}}|=g$ and $g_i\neq 0$ for all $1\leq i\leq m$. Given such vector ${\underline{g}}= (g_0,g_1\ldots,g_m)$ we denote by $\H_{\underline{g}}\subset \oH_{g,1}$ the space of tuples $(C,x_1,\alpha)$ such that:
\begin{itemize}
\item the dual graph of $C$ is the following
$$
\xymatrix@=1.5em{
*+[Fo]{g_1}\ar@{-}[rrd]& *+[Fo]{g_2} \ar@{-}[rd]  & \ldots & *+[Fo]{g_m}\ar@{-}[ld] \\ 
&&*+[Fo]{g_0},
}
$$
and the marked point $x_1$ lies on the  component of genus $g_0$.
\item  $\alpha$ is identically $0$ on the component of genus $g_0$ and non identically 0 on all the other ones.
\end{itemize}
We denote by $\oH_{\underline{g}}$ the closure of $\H_{\underline{g}}$. We introduce the following space
$$\oH_{\underline{g}}= \oH_{g_0,m+1}\times \prod_{j=1}^m \oH_{g_i,1} \hspace{8pt} \text{and} \hspace{8pt} \oM_{\underline{g}}= \oM_{g_0,m+1}\times \prod_{j=1}^m \oM_{g_i,1}.
$$
We have two natural gluing maps $\zeta_{\underline{g}} : \oM_{\underline{g}}\to \oM_{g,1}$ and $\zeta^{\#}_{\underline{g}}: \oH_{\underline{g}}\to \oH_{g,1}$ (in fact the space $\oH_{\underline{g}}$ is the pull-back of the Hodge bundle under $\zeta_{\underline{g}}$). 
\bigskip

Now, we fix a choice of $\underline{g}=(g_0,g_1\ldots,g_m)\in {\rm Part}(g)_m$ and an integer $k\geq 0$. 
\begin{mydef}\label{def:twist} A {\em twist} for the pair $(\underline{g},k)$ is a    vector $I=(i_1,\ldots, i_m)$ of positive integers such that 
$$k\geq g_0 - 1 + |I|.$$  
The {\em multiplicity} $m(I)$ of the twist $I$ is the product of its entries. 
\end{mydef} 
Given $\underline{g}$, $k$, and $I$ we construct the locus $\A_{\underline{g},k,I}\subset \P\H_{\underline{g}}$ of tuples $(C,x_1,\alpha)$ such that:
\begin{itemize}
\item the differential $\alpha$ has a zero of order $(i_i-1)$ at the node of the component of genus $g_i$ for all $1\leq i\leq m$;
\item the component $C_0$ of genus $g_0$ of the curve lies in $p(A_{g_0,(k),I})$ where $p:\oH_{g,1,I}\to \oM_{g,1}$ is the forgetful map (i.e. in the image of differentials with poles prescribed by $I$ at the nodes and a zero of order $k$  at the marked point). 
\end{itemize}
We denote by $\oA_{\underline{g},k,I}$ the Zariski closure of $\A_{\underline{g},k,I}$ in $\P\oH_{g,1}$.
\begin{mypr}[\cite{Sau}, Proposition 3.21]\label{pr:expr}
The Poincar\'e-dual class of $\oA_{\underline{g},k,I}$ is equal to 
$$ (\xi^{g_0}+ \lambda_1^0 \xi^{g_0-1}+\ldots+ \lambda_{g_0}^0)\cdot \zeta_{\underline{g}*}^{\#} \left( p^* p_*[\oA_{g_0,(k), I}]  , \prod_{j=1}^m [\oA_{g_j}(i_j-1)]  \right)$$
where $\lambda_i^0= p^*( \zeta_{\underline{g}*} (\lambda_i ,\underset{m\times}{\underbrace{1,1,\ldots, 1}}))$ (we recall that $\zeta_{\underline{g}*}$ goes from the cohomology ring $H^*(\oM_{\underline{g}})\simeq H^*(\oM_{g_0,m+1}) \bigotimes_{j=1}^{m} H^*(\oM_{g_j,1})$ to $H^*(\oM_{g,n})$).
\end{mypr}
We denote by $\alpha_{\underline{g},k,I}^0\in H^*(\oM_{g,1})$ the degree 0 coefficient (in $\xi$) of $\oA_{\underline{g},k,I}$. The above lemma implies the following expression for $\alpha_{\underline{g},k,I}^0$.
\begin{mylem}\label{lemdeg0}
The following equality holds in $H^{2k+2}(\oM_{g,1})$:
$$
\alpha_{\underline{g},k,I}^0= \zeta_{\underline{g}*} \left( \lambda_{g_0}p_*[\oA_{g_0,(k), I}] , \prod_{j=1}^m \alpha^0_{g_j}(i_j-1) \right).
$$
As a consequence $\lambda_g \cdot \alpha_{\underline{g},k,I}^0=0$ if $g_0\neq 0$.
\end{mylem}

\begin{proof}
The first part of the lemma follows from Proposition~\ref{pr:expr}. The second part is a consequence of the decomposition
$$
\lambda_g \cdot \zeta_{\underline{g}*} (1,\ldots,1)= \zeta_{\underline{g}*} \bigg(\prod_{j=0}^m \lambda_{g_j}\bigg).
$$
Indeed, this  implies the equality
\begin{equation}\label{lambdafor}
\lambda_g \cdot \alpha_{\underline{g},k,I}^0= \zeta_{\underline{g}*} \left( \lambda_{g_0}^2 p_*[\oA_{g_0,(k), I}] , \prod_{j=1}^m \lambda_{g_j} \alpha^0_{g_j}(i_j-1) \right).
\end{equation}
Now the second part of the lemma follows from Lemma~\ref{lemmum}: the class $\lambda_{g_0}^2=0$ if $g_0>0$. 
\end{proof}

\subsection{Induction formula for $a_g$'s} The main tool to compute the $a_g$'s and $d_g$'s is the following proposition.
\begin{mypr}[\cite{Sau}, Theorem 4]\label{induction}
For all $g,k\geq 0$, the following equality holds in $H^*(\oM_{g,1})$:
\begin{equation}\label{indfor}
(k+1)\psi_1 \cdot \alpha^0_g(k)= \alpha^0_g(k+1) + \sum_{m\geq 1} \frac{1}{m!} \left( \sum_{\underline{g}\in {\rm Part}(g)_m} \!\! \sum_{I} m(I) \alpha^0_{\underline{g},k,I}\right),
\end{equation}
where the right hand sum runs over all $\underline{g}$ and all possible twists for the pair $(\underline{g},k)$ as in Definition~\ref{def:twist}.
\end{mypr}

\begin{remark}
Note that our definition of twisted graph only includes trees with a unique vertex where the differentials vanishes identically. In the induction formula of~\cite{Sau} the set of twisted graphs is larger (bicolored graphs). However, bicolored graphs that are not trees do not contribute to the degree 0 part of the induction. Indeed the cohomology class associated to any bi-colored graph $\Gamma$ is of the form
 $\xi^{b_1(\Gamma)} a'$ with $a'\in H^*(\oH_{g,1})$ where 
 $b_1(\Gamma)$ is the number of loops of $\Gamma$. 
\end{remark}

The induction formula of Propostion~\ref{induction}, together with Lemmas~\ref{lambdag} and~\ref{lemdeg0} implies
\begin{myth}\label{th:reform1} We have
$$
 b_g =   \sum_{k\geq 1} \frac{1}{k!(2g-k)!} \sum_{d_1+\ldots+ d_k=g}  \left(\prod_{j=1}^k  (2d_j-1) a_{d_j} \right).
$$
\end{myth} 

\begin{remark}
This theorem is a reformulation of Formula~\eqref{mainfor} of Theorem~\ref{th:main}. Let us outline the strategy of the proof.
\begin{itemize}
\item The starting point of the proof is the induction formula of Proposition~\ref{induction} for the classes $\alpha_g^0(k)$: for all $k$, we multiply this formula by $\lambda_g$ to reduce the summation on the RHS.
\item Then we identify the numerical contribution of each term in the reduced sum for all $k$.
\item We conclude by summing the contributions for all values of $k$ between $0$ and $2g-3$. 
\end{itemize}
\end{remark}

\begin{proof}
We have seen that $a_g=  \int_{\oM_{g,1}} \lambda_g \alpha_g^0(2g-2)$. Thus we begin by multiplying equation~\eqref{indfor} by $\lambda_g$ to obtain
\begin{equation}\label{lambdapsifor}
(k+1)\psi_1 \cdot \lambda_g \alpha^0_g(k)= \lambda_g \alpha^0_g(k+1) + \sum_{\underline{g},I, g_0=0} \frac{m(I)}{(\ell(\underline{g})-1)!} \lambda_g \alpha^0_{\underline{g},k,I},
\end{equation}
where the right-hand sum runs over all $\underline{g}$ with $g_0=0$ and all possible twists. Indeed Lemma~\ref{lemdeg0} implies that the terms with $g_0\neq 0$ vanish.

Therefore we need to compute the intersection numbers
$$
\int_{\oM_{g,1}} \! \! \! \psi_1^{2g-3-k} \lambda_g \alpha^0_{\underline{g}, k, I}.
$$
We have seen (Formula~\eqref{lambdafor}) that
$$\lambda_g \cdot \alpha_{\underline{g},k,I}^0= \zeta_{\underline{g}*} \left(  p_*[\oA_{0,(k), I}] , \prod_{j=1}^m \lambda_{g_j} \alpha^0_{g_j}(i_j-1) \right).
$$
Besides the class $p_*[\oA_{g_0,(k), I}]$ is equal to $1$ if $k+1=|I|$ and $0$ otherwise. Therefore, if $k+1=|I|$ we have
$$
\int_{\oM_{g,1}} \! \! \! \psi_1^{2g-1-k} \lambda_g \alpha^0_{\underline{g}, k, I}= \left(
 \int_{\oM_{0,m+1}}\! \! \!  \psi_1^{2g-3-k} \right) \times \left( \prod_{j=1}^m \int_{\oM_{g_j,1}} \! \! \! \lambda_{g_j} \alpha^0_{g_j}(i_j-1) \right).
$$
For all $1\leq j\leq n$, the class $ \lambda_{g_j} \alpha^0_{g_j}(i_j-1)$ is not in top degree if $i_j-1\neq 2g_j-2$ thus  we get
\begin{equation} 
\int_{\oM_{g_j,1}} \! \! \! \lambda_{g_j} \alpha^0_{g_j}(i_j-1) = \left\{ \begin{array}{l l}
 0 & \text{if $i_j-1\neq 2g_j-2$}\\
 (-1)^{g_0}a_{g_0} & \text{otherwise}.
\end{array} \right.
\end{equation} 
Therefore for all $\underline{g}$ with $g_0=0$, we denote by $I(\underline{g})$ the twist $(2g_1-1, 2g_2-1,\ldots, 2g_m-1)$. 

Finally, the string equation implies that $\int_{\oM_{0,m+1}}\! \! \!  \psi_1^{2g-3-k} = 1$ if $2g-3-k=m-2$ and $0$ otherwise. Thus we set $m=2g-1-k$ and we multiply equation~\eqref{lambdapsifor} by $\psi_1^{m-2}$ to obtain
\begin{eqnarray*}
\psi^{m-2}_1 \lambda_g \left[ (k+1) \psi_1 \alpha^0_g(k)- \alpha^0_g(k+1)\right]\!\! \!\!\! &=& \!  \! \! \frac{1}{m!}  \!  \sum_{\begin{smallmatrix}
\underline{g}\in  {\rm Part}(g)_{m}\\ g_0=0
\end{smallmatrix}} \! \! \! \! \! \! \! \!  {m(I(\underline{g}))} \lambda_g  \psi^{m-1}_1  \alpha^0_{\underline{g},k,I(\underline{g})}\\
&=& \!  \! \! \frac{(-1)^{g}}{m!}\! \! \! \!  \sum_{g_1+\ldots+ g_{m}=g} \!  \left(\prod_{j=1}^{m}   (2g_j-1) a_{g_j}\right).
\end{eqnarray*}
Finally we get
\begin{eqnarray*}
\!\! (2g-2)! \psi_1^{2g-2} \!\!\!\! &\!\!\!\! & \!\!\!\! \lambda_g = a_g +\\
 & \!\!\!\! & \!\!\!\! \sum_{k>1} \! \frac{(2g-2)(2g-3)\ldots (2g-k+1)}{k!} \!\!\!\! \!\!\!\!\sum_{g_1+\ldots+ g_{k}=g} \!\! \left(\prod_{j=1}^{k}   (2g_j-1) a_{g_j}\!\! \right).
\end{eqnarray*}
This implies 
$$
 \psi_1^{2g-2} \lambda_g = \sum_{k>0} \frac{1}{k!(2g-k)!} \!\!\! \sum_{g_1+\ldots+ g_{k}=g}  \left(\prod_{j=1}^{k}   (2g_j-1) a_{g_j}\right).
$$
We use Proposition~\ref{lambdag} to conclude
$$
 b_g = \sum_{k>0} \frac{1}{k!(2g-k)!} \!\!\! \sum_{g_1+\ldots+ g_{k}=g}  \left(\prod_{j=1}^{k}   (2g_j-1) a_{g_j}\right).
$$
\end{proof}

\subsection{Induction formula for the $d_g$'s} First let us recall that $\lambda_g\cdot \delta_0=0$ (see~\cite{Mum} for example). Besides, we have the following equality 
$$
\int_{\oM_{g,1}} \delta_0 \cdot \lambda_{g-1} \psi_1^{2g-2}=\frac{1}{2}\int_{\oM_{g-1,3}} \lambda_{g-1} \psi_1^{2g-2}= \frac{1}{2} b_{g-1}
$$
(see~\cite{FabPan} for example).
Moreover, if we still denote $[\oA_g(2g-2)]=\sum_{k=0}^{2g-2} \xi^{k}\alpha_k$ then we have
$$\int_{\oA_g(2g-2)}  \xi^{2g-2} \delta_0= \int_{\oM_{g,1}} (-1)^{g-1} \delta_0\lambda_{g-1}\alpha_{0}.$$

Using these equalities, we can prove the following theorem (which is equivalent to Formula~\eqref{mainforbis} in Theorem~\ref{th:main}).

\begin{myth}\label{th:reform2}
For all $g>0$, we have
$$
\frac{b_{g-1}}{2} = \sum_{g'=1}^{g} \sum_{k=0}^{g-g'} \frac{(2g'-1) d_{g'} }{k!(2g-1-k)!} \left( \prod_{g_1+\ldots+g_k=g-g'}\!\!\! \!\!\! \!\!\!  (2g_i-1)a_{g_i}\right).
$$
\end{myth}

\begin{proof}
We follow the same strategy as in the proof of Theorem~\ref{th:reform1}: we will multiply the induction formula of Proposition~\ref{induction} by $\lambda_{g-1}\delta_0$ and then identify the numerical contributions of the sum in the RHS.

First we fix some $\underline{g}\in {\rm Part}_m(g)$. We recall that we have
$$
(\delta_0\lambda_{g-1})\cdot  \zeta_{\underline{g}}(1,\ldots,1)= \sum_{j=0}^m \zeta_{\underline{g}}(\lambda_{g_0},\ldots,\delta_0 \lambda_{g_j-1},\ldots, \lambda_{g_m}).
$$
 Now we fix $k\geq 0$ and a twist $I$ for $\underline{g}$.  We study the intersection number
$$\psi_1^{2g-3-k} \delta_0 \lambda_{g-1} \alpha_{\underline{g},k,I}^0.$$
By the same arguments as before, this intersection number vanishes if $g_0\neq 0$, or if $I\neq I(\underline{g})$, or if $2g-3-k\neq m-2$. Besides if $g_0=0$, $I=I(\underline{g})$, and $m=2g-1-k$ then
$$ \psi_1^{2g-3-k} \delta_0 \lambda_{g-1} \alpha_{\underline{g},k,I}^0= \sum_{j=1}^m (2g_j-1) d_{g_j} \left( \prod_{j'\neq j} (2g_{j'}-1)a_{g_{j'}} \right).
$$
Thus, using the induction Formula~\eqref{indfor} for the classes $\alpha^0_{g}(k)$ we get
\begin{eqnarray*}
\psi^{m-2}_1 \delta_0&\!  \! \! \!  \! \!\lambda_{g-1} \!  \! \! \!  \! \!& \left[ (k+1) \psi_1 \alpha^0_g(k)- \alpha^0_g(k+1)\right]\!\! \!\!\! 
\\&=& \!  \! \! \frac{1}{m!}\! \! \! \!  \sum_{g_1+\ldots+ g_{m}=g} \! \sum_{j=1}^m (2g_j-1) d_{g_j} \left(\prod_{j'\neq j}   (2g_{j'}-1) a_{g_{j'}}\right)\\
&=&\!  \! \! \frac{1}{(m-1)!}\!  \sum_{g'=1}^g (2g_j-1) d_{g_j} \! \! \! \!  \! \! \! \!    \sum_{g_1+\ldots+ g_{m-1}=g-g'} \! \left(\prod_{j=1}^{m}   (2g_j-1) a_{g_j}\right),
\end{eqnarray*}
if $m=2g-k-1$. Therefore, if we sum over all $0\leq k\leq 2g-3$, then we get
\begin{eqnarray*}
\!\! (2g-2)! \psi_1^{2g-2} \delta_0 \!\!\! &\!\!\!\! & \!\!\!\! \lambda_{g-1} = d_g +\\
 & \!\!\!\! & \!\!\!\! \sum_{\begin{smallmatrix}
k>0\\ 1\leq g'\leq g
\end{smallmatrix}} \! \frac{(2g-2)! }{(k)!(2g-k-1)!} (2g'-1)d_{g'} \!\!\!\! \!\!\!\!\sum_{g_1+\ldots+ g_{k}=g} \!\! \left(\prod_{j=1}^{k}   (2g_j-1) a_{g_j}\!\! \right).
\end{eqnarray*}
Finally we use the equality $\psi_1^{2g-2} \delta_0 \lambda_{g-1}= b_{g-1}/2$ to conclude.
\end{proof}

\section{Large genus asymptotics}\label{sec:asympt}

In this section, we  prove Theorem~\ref{th:asympt} using the induction formulas~\eqref{mainfor}, and~\eqref{mainforbis} and the relations~\eqref{foragvg} and~\eqref{exprSV}.  

First, let us recall that the Euler-Maclaurin formula gives the following expression for absolute values of Bernouilli numbers
$$
|B_{2g}| = \frac{2 (2g)!}{(1-2^{-g})\pi^{2g}} \sum_{k=-\infty}^{\infty} (4k+1)^{-2g}.
$$
Therefore, for all $R\in \N^*$ and for all $i \in \N^*$ we have
\begin{equation}\label{asympt:bg}
\frac{(2\pi)^{2g}}{2} b_g=  1+ O\left( \frac{1}{g^R}\right), \text{ and } \frac{(2g)!b_g}{(2g-2i)! b_{g-i}}= (2\pi)^{2i} \prod_{j=0}^{2i-1} (2g-j) + O\left( \frac{1}{g^R}\right)
\end{equation}
In particular the dominating term in the right hand side is a polynomial in $\Q[\pi^2][g]$.
Theorem~\ref{th:asympt} is a consequence of the following lemma.
\begin{mylem}\label{lem:asympt}
For all $R\in \N$, we have
\begin{eqnarray*}
(2g)! b_g \!\!\!\!\! &=& \!\!\!\!\! \sum_{i=0}^R c^{(1)}_{g,i} (2g-2i-1) a_{g-i} + O\left( \frac{(2g)!b_g}{g^{R+1}}\right),\\
(2g-1)! b_{g-1} \!\!\!\!\!  &=& \!\!\!\!\!  \sum_{i=0}^R c_{g,i}^{(2)} (2g-2i-1) d_{g-i} + \sum_{i=1}^{R} c^{(3)}_{g,i} (2g-2i-1) a_{g-i} + O\left( \frac{(2g-1)!b_{g-1}}{g^{R+1}}\right)
\end{eqnarray*}
(the last sum is $0$ if $R=0$) where $Q_g^{(i)}(x)=\sum_{k\geq 0} c^{(i)}_{g,k} x^k$ are the formal series in $x$ defined by
\begin{eqnarray*}
Q_g^{(1)}(x)&=&2g \F(x)^{2g-1},\\
Q_g^{(2)}(x)&=& 2 \F(x)^{2g-1},\\
Q_g^{(3)}(x)&=& 2(2g-1) \Delta(x) \F(x)^{2g-2}.\\
\end{eqnarray*}
\end{mylem}

\begin{proof}[End of the proof of Theorem~\ref{th:asympt} under the assumption of Lemma~\ref{lem:asympt}] 
The formal series $Q_g^{(i)}$ have coefficients in $\Q[g]$. Besides, we have $c^{(1)}_g=2g$ and $c^{(2)}_g=2$. We fix $R\geq 0$. For all $0\leq k \leq R$ we have 
\begin{eqnarray}\label{linearagbg}
(2g-2k)! b_{g-k} \!\!\!\!\! &=& \!\!\!\!\! \sum_{i=0}^{R-k} c^{(1)}_{g-k,i} (2(g-k-i)-1) a_{g-k-i} + O\left(\frac{(2g)!b_g}{g^{R+1}}\right),\\
\label{lineardgbg}
(2g-2k-1)! b_{g-k-1} \!\!\!\!\!  &=& \!\!\!\!\!  \sum_{i=0}^{R-k} c_{g-k,i}^{(2)} (2(g-k-i)-1) d_{g-k-i} \\ \nonumber + \sum_{i=1}^{R-k}\!\!\!\!\!\!    &\!\!\!\!\!   &  \!\!\!\!\!  \!\!\!\!\!   c^{(3)}_{g-k,i} (2(g-k-i)-1)a_{g-k-i}+ O\left( \frac{(2g-1)!b_{g-1}}{g^{R+1}}\right).
\end{eqnarray}
The first set of equations~\eqref{linearagbg} implies that the vector $((2g-2k)!b_{g-k})_{0\leq k\leq R-1}$ is the image of the vector $(a_{g-k})_{0\leq k\leq R-1}$ under an upper triangular matrix with coefficients in $\Q[g]$ modulo a term in 
$O\left( {(2g)!b_g}/{g^{R+1}}\right)$. The coefficients on the diagonal of this matrix are equal to $(2g-2k)(2g-2k-1)$ for all $0\leq k\leq R-1$. Therefore this linear system has an inverse in the space of matrices with coefficients in $\Q(g)$:
$$
a_g = \sum_{i=0}^{R-1} \widetilde{c}^{(1)}_{g,i} (2g-2i)! b_{g-i} + O\left( \frac{(2g)!b_g}{g^{R+1}}\right).
$$
Now we use the asymptotic behavior of the $b_{g}$ given in~\eqref{asympt:bg} to obtain
$$
a_g= \frac{2(2g-2)!}{(2\pi)^{2g}}\left(1- \frac{\pi^2}{12 g} + \frac{24\pi^2-\pi^4}{288g^2}  + \ldots + \frac{c_R}{g^R}+ O\left( \frac{1}{g^{R+1}}\right) \right)
$$
where the coefficients $c_k$ lie in $\Q[\pi^2]$. Using Proposition~\ref{vgag}, we deduce the first part of  Theorem~\ref{th:asympt}.

Now to compute the asymptotic expansion of the Siegel-Veech constants we turn to the second set of linear equations. Once again the coefficients $c_{g-k,i}^{(2)}$ define a triangular matrix with coefficients in $\Q[g]$. All coefficients on the diagonal of this matrix are equal to $2$. Therefore we have 
$$
d_g=\sum_{i=0}^{R-1} \widetilde{c}^{(2)}_{g,i} (2g-2i-1)! b_{g-i} + \sum_{i=1}^{R-1} \widetilde{c}^{(3)}_{g,i} a_{g-i}+ O\left( \frac{1}{g^{R+1}}\right)
$$
where the coefficients $\widetilde{c}^{(2)}_{g,i}$ and $\widetilde{c}^{(3)}_{g,i}$ are in $\Q[g]$. Therefore we get
$$
d_g= \frac{(2g-2)!}{(2\pi)^{2g-2}}\left(1- \frac{3+\pi^2}{12g}  + \ldots + \frac{\widetilde{c}_R}{g^R}+ O\left( \frac{1}{g^{R+1}}\right) \right)
$$
where the coefficients $\widetilde{c}_{k}$ lie in $\Q[\pi^2]$. Therefore we deduce the second part of Theorem~\ref{th:asympt} by using $c_{\rm area}(2g-2)= d_g/(4\pi^2 a_g)$.
\end{proof}

The proof of Lemma~\ref{lem:asympt} is essentially borrowed from~\cite{Ben}  but we repeat most arguments for completeness. We begin by proving two preliminary lemmas. 

\begin{mylem}\label{lem:asympt2}
Let us denote by $B_g'=(2g-1)!b_g$ (with $B_0'=0$). There exists a positive constant $C$ such that for all $g\geq k\geq2$, we have
\begin{equation}\label{ineqB}
\sum_{g_1+\ldots+ g_{k}=g} \!\!\!\! \bigg(\prod_{j=1}^{k}B_{g_i}' \bigg) \leq C^{k-1} B_{g-k+1}'.
\end{equation}
\end{mylem}

\begin{proof}
We have seen that $B_g'$ is equivalent to $2(2g-1)!/(2\pi)^{2g}$ as $g$ goes to infinity. Thus there exist positive constants $C'$ and $C$ such that
\begin{eqnarray*}
\sum_{i=1}^{g-1} B_i'B_{g-i}' &\leq & C' \sum_{i=1}^{g-1} \frac{(2i-1)! (2g-2i-1)!}{(2\pi)^{2g}} \\ 
&=&  \frac{C'}{(2\pi)^2} \frac{(2n-3)!}{(2\pi)^{2g-2}} \left(2+ \sum_{i=2}^{g-2}  \frac{2g-2}{\binom{2g-3}{2i-1}}\right) \\
&\leq &  4 \frac{C'}{(2\pi)^2} \frac{(2n-3)!}{(2\pi)^{2g-2}}  \leq  C B_{g-1}'.
\end{eqnarray*} 
We will prove that the inequalities~\eqref{ineqB} hold for all $g\geq k\geq 2$ with this constant $C$. We work by induction on $k$.  

We have seen that the inequality~\eqref{ineqB} holds for $k=2$. Suppose that $k\geq 3$. Then for all $g\geq k$ we have
\begin{eqnarray*}
\sum_{g_1+\ldots+ g_{k}=g}  \bigg(\prod_{j=1}^{k}B_{g_j}' \bigg) &=& \sum_{g_1=1}^{g-k+1}  B_{g_1}' \!\!\! \sum_{g_2+\ldots+ g_{k}=g-g_1} \bigg(\prod_{j=2}^{k-1}B_{g_j}' \bigg)\\
&\leq& C^{k-2}\sum_{g_1=1}^{g-k+1}  B_{g_1}'B_{g-g_1-k+2}' \leq C^{k-1} B_{g-k+1}',
\end{eqnarray*} 
thus  the inequality~\eqref{ineqB} holds for any $g\geq k\geq 2$. 
\end{proof}

\begin{mylem}\label{lem:asympt1}
We have the following asymptotic results
\begin{eqnarray*}
a_g &=& {2}d_g +O\left(\frac{(2g-2)! b_g}{g} \right) =  (2g-2)! b_g + O\left(\frac{(2g-2)! b_g}{g} \right).
\end{eqnarray*}
\end{mylem}

\begin{proof}[Proof of Lemma~\ref{lem:asympt1}] 
The numbers $a_g$ and $d_g$ are positive for all $g$. Therefore, using Theorems~\ref{th:reform1} and~\ref{th:reform2}, we deduce the following inequalities for all $g\geq 1$
\begin{eqnarray*}
a_g &\leq& (2g-2)! b_g,\\
2d_g &\leq& (2g-2)! b_{g-1}.\\
\end{eqnarray*}
We use these two sets of inequalities to get also lower bounds
\begin{eqnarray*}
a_g \!\!&\geq& \!\!(2g-2)! b_g - \sum_{k>1} \frac{(2g-2)!}{k!(2g-k)!} \!\!\! \sum_{g_1+\ldots+ g_{k}=g}  \left(\prod_{j=1}^{k}   (2g_j-1)! b_{g_j} \right), \\
2d_g \!\!&\geq& \!\!  (2g-2)! b_g - \sum_{
\begin{smallmatrix} k>0 \\ g_0>0 \end{smallmatrix}} \frac{(2g-2)!(2g_0-1)! b_{g_0-1}}{k!(2g-1-k)!} \!\!\!\!\!\!\!\! \sum_{g_1+\ldots+ g_{k}=g-g_0}  \left(\prod_{j=1}^{k}   (2g_j-1)! b_{g_j} \right), \\
 \!\!&\geq& \!\!(2g-2)! b_g - \sum_{k>1} \frac{(2g-2)!}{(k-1)!(2g-k)!} \!\!\! \sum_{g_1+\ldots+ g_{k}=g}  \left(\prod_{j=1}^{k}   (2g_j-1)! b_{g_j} \right), \\
\end{eqnarray*}
Now we use Lemma~\ref{lem:asympt2} to deduce
\begin{eqnarray*}
 0\leq \!\!(2g-2)! b_g - a_g\!\!&\leq&  \sum_{k=2}^g \frac{(2g-2)!}{k!(2g-k)!}  C^{k-1} B_{g-k+1}' \\
 \!\!&=& \frac{(2g-2)!}{(2\pi)^{2g}}\left[ \sum_{k=2}^g  \frac{(4C\pi^2)^{k-1}}{k!} \frac{(2g-2k+1)!}{(2g-k)!}  \right] \left(1+ o\left(\frac{1}{g}\right)\right)\\
&= & O\left( \frac{(2g-2)!}{(2\pi)^{2g}} \sum^g_{k=2}  \frac{(4C\pi^2)^{k}}{k!} \frac{1}{g} \right)\\
&=& O\left(\frac{(2g-2)!}{( 2\pi)^{2g}} \frac{1}{g}  \sum_{k>1}  \frac{(4C\pi^2)^{k}}{k!}\right) = O\left(\frac{(2g-2)!b_g}{g}\right).
\end{eqnarray*} 
By the same argument we have
\begin{eqnarray*}
 0\leq \!\!(2g-2)! \frac{b_g}{2} - d_g\!\!&\leq&  \sum_{k>1} \frac{(2g-2)!}{(k-1)!(2g-k)!}  C^{k-1} B_{g-k+1}'  \\
&=& O\left(\frac{(2g-2)!b_g}{g}\right).
\end{eqnarray*} 
\end{proof}
\begin{proof}[Proof of Lemma~\ref{lem:asympt}]
For all $R\in \N$, we have
\begin{eqnarray*}
 (2g)! b_g &=&   \sum_{k=1}^{R+1} \binom{2g}{k}  \sum_{g_1+\ldots+ g_k=g}  \left(\prod_{j=1}^k  (2g_j-1) a_{g_j} \right) \\&+&   \sum_{k=R+2}^\infty \binom{2g}{k} \sum_{g_1+\ldots+ g_k=g}  \left(\prod_{j=1}^k  (2g_j-1) a_{g_j} \right).
\end{eqnarray*}
Using Lemma~\ref{lem:asympt1}, the second sum is $O\left( \frac{(2g)!b_g}{g^{R+1}}\right)$ and the first sum is equal to 

\begin{eqnarray*}
&&\sum_{g_1=1}^{g}  (2g_1-1) a_{g_1}  \sum_{k=0}^{R} \binom{2g}{k}  \sum_{g_2+\ldots+ g_{k+1}=g}  \left(\prod_{j=2}^{k+1}  (2g_j-1) a_{g_j} \right)\\
&=& \sum_{g_1=g-R}^{g}  (2g_1-1) a_{g_1}  \sum_{k=0}^{R} \binom{2g}{k}  \sum_{g_2+\ldots+ g_{k+1}=g}  \left(\prod_{j=2}^{k+1}  (2g_j-1) a_{g_j} \right) + O\left( \frac{(2g)!b_g}{g^{R+1}}\right)\\
&=& \sum_{g_1=g-R}^g c^{(1)}_{g,g_1} (2g_1-1) a_{g_1} + O\left( \frac{(2g)!b_g}{g^{R+1}}\right)
\end{eqnarray*}
where the coefficients $c^{(1)}_{g,i}$ are defined above. This finishes the proof of the first part of Lemma~\ref{lem:asympt}. The proof of the second part is identical. 
\end{proof}


\begin{thebibliography}{10}

\bibitem{AtrEskZor}
Jayadev~S. Athreya, Alex Eskin, and Anton Zorich.
\newblock Right-angled billiards and volumes of moduli spaces of quadratic
  differentials on {$\Bbb C\rm P^1$}.
\newblock {\em Ann. Sci. \'Ec. Norm. Sup\'er. (4)}, 49(6):1311--1386, 2016.
\newblock With an appendix by Jon Chaika.

\bibitem{BCGGM}
M.~Bainbridge, D.~Chen, Q.~Gendron, S.~Grushevsky, and M.~Moller.
\newblock Compactification of strata of abelian differentials.
\newblock 2016, arXiv:1604.08834.

\bibitem{Ben}
Edward~A. Bender.
\newblock An asymptotic expansion for the coefficients of some formal power
  series.
\newblock {\em J. London Math. Soc. (2)}, 9:451--458, 1974/75.

\bibitem{CheMoeZag}
D.~Chen, , M.~Moller, and Don Zagier.
\newblock Quasimodularity and large genus limits of siegel-veech constantss.
\newblock 2016, arXiv:1606.04065.

\bibitem{EskKonZor}
Alex Eskin, Maxim Kontsevich, and Anton Zorich.
\newblock Sum of {L}yapunov exponents of the {H}odge bundle with respect to the
  {T}eichm\"uller geodesic flow.
\newblock {\em Publ. Math. Inst. Hautes \'Etudes Sci.}, 120:207--333, 2014.

\bibitem{EskOko}
Alex Eskin and Andrei Okounkov.
\newblock Asymptotics of numbers of branched coverings of a torus and volumes
  of moduli spaces of holomorphic differentials.
\newblock {\em Invent. Math.}, 145(1):59--103, 2001.

\bibitem{EskZor}
Alex Eskin and Anton Zorich.
\newblock Volumes of strata of {A}belian differentials and {S}iegel-{V}eech
  constants in large genera.
\newblock {\em Arnold Math. J.}, 1(4):481--488, 2015.

\bibitem{FabPan}
C.~Faber and R.~Pandharipande.
\newblock Hodge integrals and {G}romov-{W}itten theory.
\newblock {\em Invent. Math.}, 139(1):173--199, 2000.

\bibitem{KonZor1}
Maxim Kontsevich and Anton Zorich.
\newblock Lyapunov exponents and hodge theory.
\newblock 1997, arXiv:hep-th/9701164.

\bibitem{KonZor}
Maxim Kontsevich and Anton Zorich.
\newblock Connected components of the moduli spaces of {A}belian differentials
  with prescribed singularities.
\newblock {\em Invent. Math.}, 153(3):631--678, 2003.

\bibitem{Mas}
Howard Masur.
\newblock Interval exchange transformations and measured foliations.
\newblock {\em Ann. of Math. (2)}, 115(1):169--200, 1982.

\bibitem{MirZog}
Maryam Mirzakhani and Peter Zograf.
\newblock Towards large genus asymptotics of intersection numbers on moduli
  spaces of curves.
\newblock {\em Geom. Funct. Anal.}, 25(4):1258--1289, 2015.

\bibitem{Mum2}
D.~Mumford.
\newblock Hirzebruch's proportionality theorem in the noncompact case.
\newblock {\em Invent. Math.}, 42:239--272, 1977.

\bibitem{Mum}
David Mumford.
\newblock Towards an enumerative geometry of the moduli space of curves.
\newblock In {\em Arithmetic and geometry, {V}ol. {II}}, volume~36 of {\em
  Progr. Math.}, pages 271--328. Birkh\"auser Boston, Boston, MA, 1983.

\bibitem{Sau}
Adrien Sauvaget.
\newblock Cohomology classes of strata of abelian differentials.
\newblock arXiv:1701.07867.

\bibitem{Vee}
William~A. Veech.
\newblock Gauss measures for transformations on the space of interval exchange
  maps.
\newblock {\em Ann. of Math. (2)}, 115(1):201--242, 1982.

\bibitem{Vois1}
Claire Voisin.
\newblock {\em Hodge theory and complex algebraic geometry. {I}}, volume~76 of
  {\em Cambridge Studies in Advanced Mathematics}.
\newblock Cambridge University Press, Cambridge, 2002.
\newblock Translated from the French original by Leila Schneps.

\end{thebibliography}

\end{document}